\documentclass[reqno]{amsart}
\usepackage{xcolor}
\usepackage{tikz}
\usepackage{graphicx} % Required for inserting images
\usepackage{amsmath, amssymb, amsthm}
\theoremstyle{plain}

\newcommand\myatop[2]{\genfrac{}{}{0pt}{}{#1}{#2}}

\newcommand{\cP}{{\mathcal{P}}}
\newcommand{\cF}{{\mathcal{F}}}
\newcommand{\cM}{{\mathcal{M}}}
\newcommand{\cH}{{\mathcal{H}}}
\newcommand{\cS}{{\mathcal{S}}}
\newcommand{\R}{\mathbb{R}}

\newcommand{\N}{\mathbb{N}}
\newcommand{\Q}{\mathbb{Q}}

\newcommand{\ep}{\varepsilon}
\newcommand{\1}{\mathrm{1}}

\newcommand{\ima}{\operatorname{im}}
\DeclareMathOperator{\conv}{conv}
\DeclareMathOperator{\diam}{diam}

\DeclareMathOperator{\inter}{int}
\DeclareMathOperator{\cl}{cl}

\DeclareMathOperator{\card}{card}

\DeclareMathOperator*{\esssup}{ess\,sup}
\DeclareMathOperator{\Tan}{Tan}

\DeclareMathOperator{\di}{\diamondsuit}

\newtheorem{theorem}{Theorem}
\newtheorem{lemma}[theorem]{Lemma}
\newtheorem{corollary}[theorem]{Corollary}
\newtheorem{proposition}[theorem]{Proposition}

\theoremstyle{definition}
\newtheorem*{condition}{}

\title[(Outer) Minkowski content]{On the (outer) Minkowski content with lower-dimensional structuring element}
\author{Markus Kiderlen}
\address{Aarhus University, Department of Mathematics, Ny Munkegade 118,
DK-8000 Aarhus C, Denmark}
\email{kiderlen@math.au.dk}
\author{Jan Rataj}
\address{Charles University, Faculty of Mathematics and Physics, Sokolovsk\'a 83, 186 75 Praha 8, Czech Republic}
\email{rataj@karlin.mff.cuni.cz}

\date{\today}

\begin{document}
	
	\maketitle
	
	\begin{abstract}
		Given a convex body $Q$ (structuring element) and a set $A$ in a Euclidean space, we consider the $Q$-Minkowski content of $A$. It is defined as the usual isotropic Minkowski content of $A$, but where the Euclidean ball is replaced by $Q$. When $Q$ is full-dimensional, the existence of the $Q$-Minkowski content can be assured by a sufficient condition which was stated by Ambrosio, Fusco and Pallara in the isotropic case. If $Q$ is not full-dimensional, we show that a  weaker condition is sufficient for this purpose.  We also consider the outer $Q$-Minkowski content of $A$ yielding the anisotropic perimeter of $A$ and we find a sufficient condition for its existence. Finally, we present an example of a set in three-dimensional Euclidean space, which does not admit the isotropic outer Minkwski content, but it admits the outer $Q$-Minkowski content for all two-dimensional disks $Q$.
	\end{abstract}
	
	\section{Introduction}
	
	For two compact subsets $A$ and $Q$ of $\R^n$, the operation 
	$A\oplus Q=\{x+y:x\in A,y\in Q\}$ denotes their Minkowski sum, and we write $\lambda_n$ for the Lebesgue measure  (volume) in $\R^n$. 
	We will discuss the existence and the value of the limit 
	\begin{equation}
		\mathcal{SM}_{Q}(A)= \lim_{r\to 0_+}\frac 1{r}\lambda_n\big((A\oplus rQ)\setminus A\big), \label{eq:defQuot} 
	\end{equation}
	which will be called  the \emph{outer $Q$-Minkowski content of $A$}. The set $Q$ will be referred to as \emph{structuring element}. When $A$ is a convex body and $Q$ is the Euclidean unit ball $B^n$, $\mathcal{SM}_{B^n}(A)$ coincides with the surface area of $A$ by standard convex geometric arguments. 
	As the right-hand side of \eqref{eq:defQuot} can often be approximated algorithmically, this already indicates why the outer $Q$-Minkowski content is important in many applications when properties
    of $A$ related to the surface area are sought for. 
	In stochastic geometry, it plays a crucial role understanding the contact distribution function \cite{HL2000,JRMK2018,villa2023} and 
	digitizations of random sets \cite{KR2006}. 	
	The right-sided directional derivatives at $0$ of the covariogram $g_A:\R^n\to [0,\infty)$, \cite[Sect.~4.3]{matheron}, 
	\[
	g_A(x)=\lambda_n(A\cap(A+x))=\lambda_n(A)-\lambda_n\big((A\oplus \{x\})\setminus A\big), \quad x\in \R^n, 
	\]
	are special cases of \eqref{eq:defQuot}  and have been  considered repeatedly; see \cite{galerne}, the more general 
	\cite{JRMK2018} and the references therein.

	We will discuss the properties of the outer Minkowski content under regularity conditions on $A$ %(typically $A$ is a set with \emph{finite perimeter} $P(A)$ satisfying additional conditions)
	in the case where only very weak assumptions are imposed on the structuring element $Q$. Before doing so, we review known results on the  outer Minkowski content to put the present results into perspective. 
	
	In the early literature, the \emph{isotropic} case was considered, where the structuring element $Q$ was chosen to coincide with the Euclidean ball $B^n$ in $\R^n$. 
	Since $A\oplus rB^n$ 
	is the set of all points in $\R^n$ with Euclidean distance at most $r$ from $A$, it is natural to consider anisotropic distances as well, corresponding to structuring elements that are 
	convex bodies (compact convex subsets of $\R^n$ with interior points),
	origin-symmetric  and with their interior containing the origin $0$. 
	Apart from the origin-symmetry, which is dispensable here, these are the standard assumptions on $Q$ in much of the existing literature, but in the present paper, $Q$ may also be lower-dimensional and even non-convex. Exceptions from this rule are \cite{KR2006} and \cite{JRMK2018} where the structuring element is chosen to be finite. %, and \cite{} (I remeber a work where a local metric was used.)

The outer $Q$-Minkowski content is closely related to the $(n-1)$-dimensional  \emph{$Q$-Minkowski content} which is defined for $E\subset\R^d$ as
\begin{equation}  \label{def_M}
\cM_Q(E):=\lim_{r\to 0_+}\frac 1{2r}\lambda_n(E\oplus rQ),
\end{equation}
provided that the limit exists (cf.\ \cite{LV16}). An easy comparison with \eqref{eq:defQuot} reveals that 
$\cM_Q(E)$ equals $\frac 12 \mathcal{SM}_{Q}(E)$ if $\lambda_n(E)=0$ and $\infty$ otherwise.
Its isotropic version with $Q=B^n$ yields the classical Minkowski content which is known to agree with the $(n-1)$-dimensional Hausdorff measure for sets $E$ which are ``nice enough'' $(n-1)$-dimensional sets  in the sense of geometric measure theory -- we will list several sufficient criteria below.

While the (classical) Minkowski content is naturally compared with the Hausdorff measure, the \emph{outer} Minkowski content is related to the \emph{perimeter}. Consider a set $A\subset\R^n$ of finite perimeter $P(A)$. Ambrosio, Colesanti and Villa \cite[Theorem~5]{ACV08} showed that if 
its boundary $\partial A$ admits the isotropic Minkowski content $\mathcal M_{B^n}(\partial A)$ and if it coincides with $P(A)$, then 	$\mathcal{SM}_{B^n}(A)$ exists and has the same value. 
Intuitively, in the non-isotropic case, the contribution of a boundary patch to the local parallel volume must be proportional to the `height' of the structuring element $Q$ in the outer normal direction of the patch, so one expects a relation of the form
	\begin{equation}\label{eqChamb}
		\mathcal{SM}_{Q}(A)= \int_{S^{n-1}} h(Q\cup\{0\},v)S_{n-1}^*(A,dv)=:P_Q(A)  
	\end{equation}
involving the  \emph{anisotropic perimeter $P_Q(A)$ of $A$}. Here,
$h_M(v)=\max_{x\in M}\langle x,v\rangle$ is the support function of the convex hull of a set $M\subset \R^n$, and  
	$S_{n-1}^*(A,\cdot)$ is the \emph{generalized surface area measure}. This measure is defined in Section \ref{sec:notation} below and can be thought of as 
	the image measure of $\cH^{n-1}$,  restricted to the essential boundary $\partial^* A$ of $A$, under the (a.e.~existing) Gauss map for the  \emph{outer} normal.  
    
    Chambolle at al.~\cite[Thm.~3.4]{CLL14} showed that  \eqref{eqChamb} indeed holds true  under the standard assumptions on $Q$ if $A$ is a set with finite perimeter such that the isotropic $\mathcal{SM}_{B^n}(A)$ exists and coincides with $P(A)$. We will show in  Theorem~\ref{prop0}~(ii), below, that the same statement holds true for general nonempty compact $Q$; this setting is the reason why $Q\cup\{0\}$ instead of $Q$ appears in  \eqref{eqChamb}. 
    
    We also extend the above mentioned result \cite[Theorem~5]{ACV08},
    replacing $B^n$ with an arbitrary  (possibly lower-dimensional) convex body $Q$: if 
the boundary $\partial A$ admits the anisotropic $Q$-Minkowski content $\mathcal M_Q(\partial A)$ and if it coincides with $P_{\di Q}(A)$, where  $\di Q:=\frac 12(Q\oplus(-Q))$ is the \emph{symmetral} of $Q$, then $\mathcal{SM}_Q(A)$ exists and equals $P_Q(A)$, see Proposition~\ref{prop_AFP}.
	
	%When does the (outer) Minkowski content exist?  
	A classical result by Federer {\cite[Theorem~3.2.29]{Federer69}} states that $\mathcal M_{B^n}(E)$ {exists} and equals $\cH^{n-1}(E)$ if the closed set $E\subset \R^n$  is $(n-1)$-rectifiable (i.e.~it is the Lipschitz image of a bounded subset of $\R^{n-1}$). If the compact set 
	$E$ is merely countably $\cH^{n-1}$-rectifiable (i.e.~it can, up to a $\cH^{n-1}$-null set be covered by a countable union of Lipschitz images of $\R^{n-1}$) its isotropic Minkowski content need not exist, see the counterexample \cite[Ex.~2.103]{AFP}. This is why Ambrosio, Fusco and Pallara \cite{AFP} introduced the following condition, which will be given a name for further reference. 
	
	\begin{condition}[\textbf{AFP-condition}] 
		There exists a Radon measure $\nu$ on $\R^n$, absolutely continuous with respect to $\cH^{n-1}$,  and a constant $\gamma>0$, such that 
		\[
		\nu(B(x,r))\geq\gamma r^{n-1},\quad x\in E,\, r\in(0,1). 
		\]
	\end{condition}

	It is shown in  \cite{AFP} that the isotropic Minkowski content exists and coincides with $\cH^{n-1}(E)$ , if the compact, countably $\cH^{n-1}$-rectifiable set  $E\subset \R^n$  satisfies the AFP-condition. Lussardi and Villa \cite[Thm.~3.4]{LV16}
	proved that the same conditions also imply the existence of 
	$\mathcal{M}_{Q}(E)$   and allows to express this limit geometrically %\mk{as an anisotropic perimeter (see below)} 
    under standard assumptions on $Q$. 
	We extend their result not only allowing for any  compact convex set $0\in Q\subset \R^n$, but replacing the 
	AFP-condition by the following weaker one,  where $L$ is chosen as the linear span of $Q$. 
	
	\begin{condition}[\textbf{AFP-condition relative to $L$}]
		There exists a Radon measure $\nu$, absolutely continuous with respect to $\cH^{n-1}$, and a  constant $\gamma>0$, such that  
		\begin{equation}\label{suff_k}
			\nu(B(x,r))\geq\gamma r^{k-1}\lambda_{n-k}(p_{L^\perp}(E\cap B(x,r))),\quad x\in E,\, r\in(0,1).
		\end{equation}
		Here, $p_{L^\perp}$ denotes the orthogonal projection onto the orthogonal complement $L^\perp$ of $L$.

	\end{condition}
In Theorem \ref{TM} below we show that if the compact set $E\subset\R^n$ is  countably $\cH^{n-1}$-rectifiable and  satisfies the AFP-condition relative to $L$, then it admits the $Q$-Minkowski content and we have
$$\cM_Q(E)=\int_Eh(\di Q,\nu_E(x))\, \cH^{n-1}(dx),$$
where $\nu_E(x)$ denotes a unit vector perpendicular to the approximate tangent space $\Tan^{n-1}(E,x)$ and is defined up to sign $\cH^{n-1}$-almost everywhere on $E$ (see Subsection~\ref{GMT}).

The AFP-condition has also been used  in \cite{LV16} to show the existence of the outer $Q$-Minkowski content of a set $A$ of finite perimeter when $Q$ is full dimensional. Lussardi and Villa show that if $\partial A$ is compact, countably $\cH^{n-1}$-rectifiable and satisfies the AFP-condition then $A$ has the outer $Q$-Minkowski content for any full-dimensional convex body $Q$, and this equals
$$\mathcal{SM}_Q(A)=P_Q(A)+2\int_{\partial A\cap A^0}h({\di Q},\nu_{\partial A}(v))\, \cH^{n-1}(dv),$$
where $A^0$ denotes the set of points with vanishing Lebesgue density of $A$ (see \cite[Theorem 4.4]{LV16}). We avoid the contribution of the lower-dimensional part of $\partial A$ to the outer Minkowski content by assuming that $P(A)=\cH^{n-1}(\partial A)$; this implies that the integral in the above formula vanishes. Under this constraint, we extend their result to lower-dimensional convex $Q$ in Corollary~\ref{C1}.

	It should be noted that the existence of 	$\mathcal{SM}_{B^n}(A)$, and, a fortiori, the AFP-condition, are not necessary for the existence of the outer Minkowski content with a 
    lower-dimensional structuring element.  Indeed, Example 3 in Section \ref{Sect_examples} presents a compact set $A\subset \R^3$ of finite perimeter that does not allow for an isotropic outer Minkowski content, although $\mathcal{SM}_{B^3\cap L}(A)$ exists for all two-dimensional subspaces $L\subset \R^3$. 
    As the construction is difficult,  we also give a simple set $A\subset \R^n$ without isotropic outer Minkowski content that does allow for $\mathcal{SM}_{B^n\cap L}(A)$ for at least one linear subspace $L$ of given dimension in $\{2,\ldots,n-2\}$, see Example 1. 
    Example 2 presents a compact set of finite perimeter in $\R^3$, for which $\mathcal{SM}_{Q}(A)$ does not exist for $Q=B^3$, nor for any convex set $Q$, such that its affine hull is a two-dimensional subspace of $\R^3$. 
	
	Even if $\mathcal{SM}_Q(A)$ exists, \eqref{eqChamb} will generally not hold, as the right-hand side does not depend on changes of $A$ on a $\lambda_n$-null set, whereas the left-hand side can be made arbitrarily large, for instance by adding a  large  $(n-1)$-dimensional disc to $A$ when $Q$ has interior points. 
	
Even if additional assumptions on $A$ imply that this extra term vanishes, for instance by requiring that $A\cap A^0=\emptyset$, it is an unresolved problem to find a necessary and sufficient geometric condition 
on $A$ to guarantee the existence of $\mathcal{SM}_{Q}(A)$ of the form \eqref{eqChamb}. 
However, in the case of the structuring element being a line segment $Q=[0,u]$, $u\in S^{n-1}$,  we prove the surprising result that ${\mathcal{SM}}_{[0,u]}(A)$ exists for all sets $A\subset \R^n$ of finite perimeter if $A\cap A^0=\emptyset$, and that \eqref{eqChamb} holds in this case. More explicitly, Theorem \ref{ThmQlineSegment} shows 
\[ 
{\mathcal{SM}}_{[0,u]}(A)=\int_{S^{n-1}}\langle u,v\rangle^+\, S_{n-1}^*(A^*,dv).
\]
Here, we used the notation $t^+=\max\{0,t\}$, $t\in\R$. This result is obtained by choosing a suitable representative 
of $A$ in the family $[A]_\sim$  of all sets that coincide with $A$ up to a $\lambda_n$-null set. A
similar approach was used in \cite{CLL14}. 

The paper is organized as follows. After having introduced notation and preliminary results in Section \ref{sec:notation}, 
we consider and dissuss the outer Minkowski content as a function of $L^1$-classes of sets in Section \ref{sec:3}. Section \ref{sec:4}  treats results for general structuring elements $Q$ that do not require a version of the AFP-condition,  and Section \ref{sec:5} is dedicated to one-dimensional structuring elements. The central results in Section \ref{sec:suff_cond}	 explain the role of the AFP-condition relative to $L$ to assure  the existence of $Q$-Minkowski content (Theorem \ref{TM}) and outer $Q$-Minkowski content 
(Corollary \ref{C1}) for convex $Q$ parallel to some subspace $L\subset \R^n$. The paper concludes with three examples in Section~\ref{Sect_examples}, where the last of these examples is one where the  AFP-condition relative to $L$ can explicitly be confirmed, whereas the ordinary AFP condition is violated.

We are very grateful to Filip Fry\v s for carefully reading the manuscript and pointing out several typos and errors.

\section{Notation and Preliminaries}
\label{sec:notation}
For sets $A,B\subset \R^n$, the set $A\Delta B=(A\setminus B)\cup (B\setminus A)$ denotes their symmetric difference. 
Our notations for the boundary, the closure and the interior of $A$ are $\partial A$, $\cl A$ and $\inter A$, respectively. 
We will write $\langle x,y\rangle$ for the usual inner product of $x,y\in \R^n$ and  $\|x\|$ for the Euclidean norm of $x$. 
$B^n$ denotes the Euclidean unit ball, $S^{n-1}=\partial B^n$  its boundary and $\kappa_n=\lambda_n(B^n)$ its volume. 
The two interchangeable notations for the ball of radius $r\ge 0$ with center $x\in \R^n$ are $B(x,r)=x+rB^n$. 
Given a linear subspace $L\subset \R^n$, we will write $p_L:\R^n\to L$ 
for the orthogonal projection onto $L$.  We will also use the notation $B^k$, $k\in \{1,\ldots,n-1\}$, for a $k$-dimensional unit ball in $\R^n$, meaning that there is a $k$-dimensional linear subspace $L$ in $\R^n$ such that $B^k=B^n\cap L$.

\subsection{Geometric measure theory}  \label{GMT}
We refer the reader to \cite{AFP} or \cite{Federer69} concering notions and results from geometric measure theory.	For every measurable set $A\subset \R^n$ the Lebesgue density 
at $x\in \R^n$ is given by
\[
\Theta^n(A,x)=\lim_{r\to 0_+} \frac{\lambda_n(A\cap (x+rB^n))}{\lambda_n(x+rB^n)},
\]
if the limit exists. For $t\in [0,1]$ we let 
\[
A^t=\{x\in \R^n: \Theta^n(A,x)=t\}. 
\]

Let $k\in\{0,1,\dots,n\}$ be given. A set $E\subset\R^n$ is called \emph{$k$-rectifiable} if it is a Lipschitz image of a bounded subset of $\R^k$, and \emph{countably $\cH^k$-rectifiable} if we can write $E=E_0\cup E_1\cup\dots$ with $k$-rectifiable sets $E_1,E_2,\dots$ and $\cH^k(E_0)=0$. If $E\subset\R^n$ is countably $\cH^k$-rectifiable then for $\cH^k$-almost all $x\in E$, the cone $\Tan^k(E,x)$ of \emph{approximate tangent vectors} of $E$ at $x$ forms a $k$-dimensional vector subspace of $\R^n$. In particular, if $k=n-1$, we will write $\nu_E(x)$ for a unit vector perpendicular to $\Tan^{n-1}(E,x)$ and call it a \emph{unit normal vector to $E$ at $x$}; the vector $\nu_E(x)$ is defined uniquely up to sign at $\cH^{n-1}$-almost all $x\in E$.

We will later use  the following property (which follows easily from the above mentioned facts): If $E\subset\R^n$ is countably $\cH^k$-rectifiable and $E_0\subset E$ then $E_0$ is countably $\cH^k$-rectifiable as well and we have $\Tan^k(E_0,x)=\Tan^k(E,x)$ for $\cH^k$-almost all $x\in E_0$. In the case $k=n-1$ we have also $\nu_{E_0}(x)=\pm\nu_E(x)$ for $\cH^{n-1}$-almost all $x\in E_0$.

\subsection{Sets with finite perimeter}
By definition, cf.~\cite{AFP}, a measurable set $A\subset\R^n$ has finite perimeter, if its 
indicator  function $1_A$ has a 
distributional derivative that can be represented 
as a finite Radon measure $D\1_A$. (Poly-)convex sets, compact sets with positive reach
%UPR-sets 
and compact Lipschitz domains are sets of finite perimeter.	If $A$ is a set of finite perimeter, then the variation measure $|D\1_A|$ can be written as a restriction of the $(n-1)$-dimensional Hausdorff measure $\cH^{n-1}$ in the form 
\begin{align} \label{D1_A_Interpretation}
	|D\1_A|=\cH^{n-1}\llcorner (\partial^* A),
\end{align}
where $\partial^* A=\R^n\setminus (A^0\cup A^1)$ is the  \emph{essential  boundary}  of  $A$, see \cite[(3.63)]{AFP}. The  \emph{perimeter} of $A$ is 
\begin{equation}  \label{perimeter}
	P(A)=|D\1_A|(\R^n)=\cH^{n-1}(\partial^* A).
\end{equation}
The polar decomposition (cf.\ \cite[Corollary 1.29]{AFP}) of  $D1_A$ reads 
$$D1_A=\Delta_{1_A}\,|D1_A|,$$
where $\Delta_{\1_A}$ is an $S^{n-1}$-valued function defined
$\cH^{n-1}$-almost everywhere on $\partial^*A$ and can be interpreted
as a generalized inner unit normal vector field to $A$. (In fact there
exists a subset of $\partial^*A$ of full $\cH^{n-1}$ measure, called
{\it reduced boundary} $\cF A$ and a representative $\nu_A$ of $-\Delta_{1_A}$
defined there such that the half-space $\{y:\, \langle y,\nu_A(a)\rangle\leq 0\}$
coincides with the approximate tangent cone of $A$ at $a$ for any $a$
from the reduced boundary, see \cite[\S3.5]{AFP}.) Thus, it is 
natural to define the {\it generalized surface area measure} of a set $A$ with finite perimeter as
\begin{equation}  \label{def_SAM}
	S_{n-1}^*(A;\cdot)=\cH^{n-1}\{a\in\partial^*A:\, -\Delta_{\1_A}(a)\in\cdot\}.
\end{equation}
The measure $S_{n-1}^*(A;\cdot)$ coincides with the usual surface area measure  $S_{n-1}(A;\cdot)$ if $A$ has Lipschitz boundary.

If $A\subset \R^n$ is a set of finite perimeter then 
$A^C$ has also finite perimeter  and $\cP(A)=\cP(A^C)$. Even stronger, $|D1_A|=|D1_{A^C}|$ (from,  e.g., \cite[Theorem 3.59]{AFP} and the fact that $\cH^{n-1}\llcorner \cF A=\cH^{n-1}\llcorner \cF A^C$).
This implies $\nu_A(x)=-\nu_{A^C}(x)$ for a.e. $x\in \cF A$ and 
\begin{equation}
	\label{eq_S_AC}
	S^*_{n-1}(A^C,S)=S^*_{n-1}(A,-S)
\end{equation}
for all Borel sets $S\subset S^{n-1}$ follows. 
We define the \emph{anisotropic perimeter} of $A$ with respect to a structuring element $Q$ as
$$P_Q(A)=\int_{S^{n-1}}h(Q\cup\{0\},v)\, S_{n-1}^*(A,dv).$$
Note that $P_{B^n}(A)=P(A)$ (since $h(B^n,\cdot)=1$) and, due to \eqref{eq_S_AC}, we have $P_Q(A^C)=P_{-Q}(A)$.
\medskip

\section{Outer Minkowski content of $L_1$-classes of sets}\label{sec:3}
Let $A\subset\R^n$ have finite perimeter and $Q\subset\R^n$ be a nonempty compact set.
Following the notation from \cite{JRMK2018}, we define
\begin{equation}
	\label{eq_defG}
	G(Q,\1_A):=\int_{\R^n}\left(\sup_{u\in Q}\1_A(x-u)-\1_A(x)\right)^+\, \lambda_n(dx)=\lambda_n((A\oplus Q)\setminus A).
\end{equation}
Note that we have the alternative representation  
$$G(Q,\1_A)=\int_{A^C}\1_{A\cap(x-Q)\neq\emptyset}\, \lambda_n(dx).$$

Later, we will need the fact that if $K$ is a (full-dimensional) convex body then 
\begin{equation} \label{E_K_int}
	\lambda_n(A\oplus K)=\lambda_n(A\oplus \inter K)
\end{equation}
holds for any measurable set $A\subset\R^n$. 
Note first that it is sufficient to show \eqref{E_K_int} for compact sets $A$ and then use the approximations where the set $A$ above is replaced with $(\cl A)\cap B(0,R)$ and $R\to\infty$. For  compact $A$, relation \eqref{E_K_int} is a consequence of the continuity of the volume function $r\mapsto \lambda_n(A\oplus rK)$, which follows from the results of Chambolle et al.\ in \cite{CLV21} (indeed, \cite[Theorem~5.2]{CLV21} states that the volume function even has finite 
one-sided derivatives at any $r>0$).

We now turn to the problem that \eqref{eq:defQuot}  depends on the representative of $A$ in $[A]_\sim$.  
If the  structuring element $Q$ is at most countable--like in  \cite{JRMK2018}, where it was assumed to be finite--the value in \eqref{eq_defG} is unchanged if $A$ is altered on a $\lambda_n$-null set.
As this is not the case for all compact $Q$, we choose the  representative  $\underline A:=(A^0)^C\in [A]_\sim$, and define 
\[
\underline G(Q,\1_A):=G(Q,\1_{\underline A}). 
\]
Since, for any $B\in [A]_\sim$ we have $\underline B=\underline A$, it follows that $B\mapsto \underline G(Q,\1_B)$ is constant on $[A]_\sim$. %Also, $\underline G(Q,\1_A)= G(Q,\1_A)$ when $Q$ is at most countable. 
The following lemma gives more explicit representations of $\underline G(Q,\1_A)$ if $Q$ is a convex set.

\begin{lemma}\label{G*}
	Let a measurable set $A\subset\R^n$ %with finite perimeter 
	be given and fix a nonempty compact convex set $Q$ of dimension $k\in \{0,\ldots,n\}$.
	\begin{enumerate}
		\item[(i)] We have
		\begin{align}\nonumber
			\underline G(Q,\1_A)=&\int_{\R^n}
			\left(\esssup^{k}_{u\in Q}
			\1_A(x-u)-\1_A(x)\right)^+\, \lambda_n(dx)\\
			=&\int_{\R^n}\left(\1_{\lambda_k(A\cap(x-Q))>0}-\1_A(x)\right)^+\, \lambda_n(dx),
			\label{eq:G*}
		\end{align}
		where the essential supremum is understood with respect to $\lambda_k$ in the affine hull of $Q$.
		\item[(ii)] 
		$\underline G(Q,\1_A)=\min\big\{G(Q,\1_{B}):
		\,B\in [A]_\sim\big\}$.
		\item[(iii)]
		If $A\cap A^0=\emptyset$, then 
		\[
		\underline G(Q,\1_A)=G(Q,\1_A). 
		\]
	\end{enumerate}
\end{lemma}

\begin{proof}
	In order to prove (i), assume first that $Q$ is full-dimensional, i.e., $k=n$. We have for any $x\in\R^n$
	$$\underline A\cap(x-\inter Q)\neq\emptyset\implies\lambda_n(A\cap(x-Q))>0\implies \underline A\cap(x-Q)\neq\emptyset,$$
	where the first implication follows from the definition of the density and the second one from the Lebesgue density theorem. Integration over $A^C$ gives
	\begin{align*}
		\int_{A^C}
		\1_{\underline A\cap(x-\inter Q)\neq\emptyset}\,\lambda_n(dx)
		&\leq \int_{A^C}\1_{\lambda_n(A\cap(x-Q))>0}\,\lambda_n(dx)
		\\&\leq\int_{A^C}\1_{\underline A\cap(x-Q)\neq\emptyset}\,\lambda_n(dx),			
	\end{align*}
	which can be rewritten as
	\begin{align*}
		\lambda_n((\underline A\oplus\inter Q)\setminus A)&\leq 
		\int_{\R^n}\left(\esssup^{n}_{u\in Q}
		\1_A(x-u)-\1_A(x)\right)^+\,dx\\&
		\leq\lambda_n((\underline A\oplus Q)\setminus A).
	\end{align*}
	Using \eqref{E_K_int}, we obtain (i) for $k=n$.

	If $k<n$ let $L$ be the $k$-subspace parallel to $Q$ and $L^\perp$ its orthogonal complement. When intersecting $A$ and a translation $L+y$ of $L$ with some $y\in L^\perp$, we will denote by $\underline{A\cap(L+y)}$ the set of all points in $y+L$ that do not have a vanishing $k$-dimensional density $\Theta^k(A\cap(L+y),\cdot)$ in $L+y$.
	Lebesgue's density theorem, applied in $L+y$, implies that
	$\lambda_k\big(\underline{A\cap(L+y)}\,\Delta\,(A\cap(L+y))\big)=0$ for all $y\in L^\perp$. 
	Fubini's theorem thus  yields that for $\lambda_{n-k}$-almost all $y\in L^\perp$,
	\[
	\lambda_k\big(\underline{A\cap(L+y)}\,\Delta\,(\underline A\cap(L+y))\big)=0, 
	\]
	where we also used $\underline A\in [A]_\sim$.
	Hence, applying the first part of the proof to the set $A\cap(L+y)\subset L+y$, we have
	\begin{align*}
		\int_{A^C\cap(L+y)}&\1_{
			\lambda_k(A\cap(L+y)\cap(x-Q))>0}\,\lambda_k(dx)\\
		&=\int_{A^C\cap(L+y)}\1_{\underline{(A\cap(L+y))}\cap(x-Q)\neq\emptyset}\, \lambda_k(dx)\\
		&=\int_{\underline A^C\cap(L+y)}\1_{\underline A\cap(L+y)\cap(x-Q)\neq\emptyset}\, \lambda_k(dx)
	\end{align*}
	for $\lambda_{n-k}$-almost all $y\in L^\perp$. Integrating over $y\in L^\perp$ we get (i).

	We show (ii). If $\lambda_k(A\cap (Q-x))>0$, then we certainly have $A\cap (Q-x)\ne \emptyset$, so \eqref{eq:G*} implies $\underline G(Q,\1_A)\leq G(Q,\1_A)$. Hence, for any set $B\in [A]_\sim$ we have $G(Q,\1_B)\ge \underline G (Q,\1_B)=\underline G(Q,\1_A)$, implying  (ii).  
	
	To prove (iii) note that $A\cap A^0=\emptyset$  implies $A\subset \underline A$, so $G(Q;1_A)\le \underline G(Q;1_A)$, whereas the reverse inequality follows from (ii). 
\end{proof}
Lemma \ref{G*}~(i) implies that $\underline G(Q;1_A)$ generalizes a similar concept in \cite{CLL14}, where only convex structuring elements with the origin as interior point have been considered. Their $\mathcal{SM}_{r,Q}(A;\R^n)$ coincides with our $\tfrac1r \underline G(rQ;1_A)$.

\section{Results for general structuring elements}\label{sec:4}

The following proposition extends known results for 
finite and full-dimensional structuring elements to the case of nonempty compact structuring elements.
\begin{theorem}\label{prop0}
	Let $A\subset \R^n$ be a set with finite perimeter. 
	For a nonempty compact set $Q\subset\R^n$ the following two statements hold. 
	\begin{enumerate}
		\item[(i)] We have 
		\begin{equation}\label{eq:trivialLowerBound}
			\liminf_{r\to 0_+}\frac1r\lambda_n((A\oplus rQ)\setminus A)\ge P_Q(A).
		\end{equation}
		\item[(ii)] If the isotropic outer Minkowski content of $A$ exists and satisfies $\mathcal {SM}(A)=\mathcal P (A)$, then $A$ admits an outer $Q$-Minkowksi content and 
		\begin{equation}\label{eq:if_iso_OMC}
			\mathcal{SM}_{Q}(A)=P_Q(A).
		\end{equation}
	\end{enumerate}
\end{theorem}
\begin{proof}
	We show (i). 
	For any nonempty finite set $\tilde Q\subset Q$,  \cite[Theorem 1]{JRMK2018} implies 
	\begin{align*}
		\liminf_{r\to 0_+}\frac1r\lambda_n((A\oplus r Q)\setminus A)
		\ge \liminf_{r\to 0_+}\frac1r\lambda_n((A\oplus r\tilde Q)\setminus A)
		={P_{\tilde{Q}}(A)}. 
	\end{align*}
	Since $S^*_{n-1}(A,\cdot)$ has finite total mass $P(A)<\infty$, dominated convergence implies 
	\[
	\liminf_{r\to 0_+}\frac1r\lambda_n((A\oplus rQ)\setminus A)\ge 
	\int_{S^{n-1}} 
	\sup_{ \myatop{\emptyset \ne \tilde Q\subset Q }                            {\tilde Q\text{ finite }}
	}
	h(\tilde Q\cup\{0\},v) S^*_{n-1}(A,dv) . 
	\]
	Since $h(Q\cup\{0\},\cdot)=h(\conv (Q\cup\{0\}),\cdot)$ and $\conv (Q\cup\{0\})$ can be approximated by convex polytopes in the Hausdorff-metric, the integrand equals $h(\conv (Q\cup\{0\}),\cdot)=h(Q\cup\{0\},\cdot)$, and the first claim is shown. 
	
	To prove (ii), the assumption $\mathcal {SM}(A)=\mathcal P (A)$ allows us to apply  the already mentioned result 
	\cite[Theorem 3.4]{CLL14} with the convex body $C=\conv (Q\cup \{0\})+\varepsilon B^n$. One obtains
	\begin{align*}
		\limsup_{r\to 0_+}\frac1r\lambda_n((A\oplus r Q)\setminus A)&\le \limsup_{r\to 0_+}\frac1r\lambda_n((A\oplus r C)\setminus A)
		\\&=\int_{S^{n-1}} h(C,v) S^*_{n-1}(A,dv).
	\end{align*}
	Since $h(C,\cdot)=h(Q\cup \{0\},\cdot)+\varepsilon$ and 
	$S^*_{n-1}(A,dv)$ has finite total mass, we may let $\varepsilon\to 0_+$ and obtain 
	\begin{align*}
		\limsup_{r\to 0_+}\frac1r\lambda_n((A\oplus r Q)\setminus A)&\le \int_{S^{n-1}} h(Q\cup \{0\},v) S^*_{n-1}(A,dv)=P_Q(A). 
	\end{align*}
	Together with statement (i), this implies (ii). 
\end{proof}

\begin{corollary}  \label{coro1}
Let  $A\subset\R^n$ be  set of finite perimeter and let $Q$ be a nonempty compact subset of $\R^n$. 
If $\partial A$ admits the isotropic Minkowski content which, in addition, coincides with $P(A)$, then $A$ admits the outer $Q$-Minkowski content and 
	$$\cS\cM_Q(A)=P_Q(A). $$
\end{corollary}
\begin{proof}
	The assumptions on the set $A$ imply that 
	it
	admits the outer Minkowski content $\cS\cM(A)$ and that $\cS\cM(A)=P(A)$ by \cite[Theorem~5]{ACV08}. Theorem \ref{prop0}~(ii) now yields the claim. 
\end{proof}

\section{Line segments as structuring elements }\label{sec:5}

Given two points in $\R^n$, one belonging to a set $A$ and the other not, the segment connecting the two points must hit the topological boundary of $A$. This easy fact has its counterpart for sets with finite perimeter and their essential boundary. In particular, if $A\subset\R^n$ has finite perimeter, $u\in S^{n-1}$ and $s>0$ then the set
\begin{equation} \label{E_Ms}
	M_s(A):=\{x\in A:\, x+su\not\in A,\, \cF_{u+}A\cap [x,x+su]=\emptyset\}
\end{equation}
has Lebesgue measure zero, see \cite[Lemma~5~(ii)]{JRMK2018}. 
Here, $\cF_{u+}A$ is the set of all $a\in \cF A$ with $\langle \nu_A(a),u\rangle >0$. We will need the following consequence of this fact.

\begin{lemma}  \label{L1}
	If $A\subset\R^n$ has finite perimeter and $u\in S^{n-1}$ then there is a set $A^*\in [A]_\sim$ such that 
	\begin{equation}\label{eq:Agood}
		\lambda_n\big(
		(A^*\oplus r[0,u])
		\setminus(A^*\cup\bigcup_{a\in \cF_{u+}A}[a,a+ru])
		\big)=0
	\end{equation}
	for all $r>0$. 
\end{lemma}

\begin{proof}
	From \cite[Lemma~3.1]{R15} we infer that $A\cap p_{u^\perp}^{-1}\{z\}$ is a one-dimensional set of positive perimeter for $\lambda_{n-1}$-almost all $z\in u^\perp$. Thus, using \cite[Proposition~3.52]{AFP}, 
	$A\cap p_{u^\perp}^{-1}\{z\}$ equals, up to $\lambda_1$-measure zero, a locally finite disjoint union of nonempty open segments $J_1^z,J_{2}^z,\ldots$, say. 
	
	Based on an idea in \cite[Sect.~4.1]{chlebik} we show that there is a set $A^{**}\in [A]_\sim$ such that 
	\begin{equation}\label{eq_A**}
		A^{**}\cap p_{u^\perp}^{-1}\{z\}=\bigcup_i J_i^z
	\end{equation}
	for $\lambda_{n-1}$-almost all $z\in u^\perp$. To show \eqref{eq_A**}, we  will assume that $u$ is the $n$th standard basis vector  $e_n$ of $\R^n$ in order to keep notation concise. For given $a<b$ define the set 
	\[
	A_{a,b}:=\{ z\in e_n^\perp: \cH^1\big((\{z\}\times (a,b))\setminus A\big)=0\}. 
	\]
	We have $(\{z\}\times (a,b))\setminus A=p_{e_n^\perp}^{-1}(z)\cap (e_n^{\perp}\times (a,b))\setminus A$, so $z\mapsto \cH^1\big((\{z\}\times (a,b))\setminus A\big)$ is measurable as a function  of
	$z\in e_n^\perp$ by Tonelli's theorem, implying the measurablity of 
	$A_{a,b}$ in $e_n^\perp$. Thus also 
	\[
	A^{**}:=\bigcup_{\genfrac{}{}{0pt}{}{a<b}{a,b\in \Q}}A_{a,b}\times (a,b)
	\]
	is measurable. Clearly, \eqref{eq_A**} holds 
	for $\lambda_{n-1}$-almost all $z\in u^\perp$ and another application of Tonnelli's theorem now imples $A^{**}\in [A]_\sim$. 
	
	The set 
	\[
	A^*=A^{**}\setminus\bigcup_{s\in \Q\cap (0,\infty)}M_s(A^{**})
	\]
	is an element of $[A]_\sim$ due to \eqref{E_Ms}. In view of \eqref{eq_A**}, 
	$\lambda_n$-almost any point $x\in (A^*\oplus r[0,u])
	\setminus(A^{**}\cup A)$ must satisfy 
	\[
	\emptyset\ne (\cF_{u+}A^{**})\cap [x,x+ru], 
	\]
	and \eqref{eq:Agood} now follows from the fact that $\cF_{u+}A^{**}=\cF_{u+}A$.
\end{proof}

\begin{theorem}\label{ThmQlineSegment}
	Let $A\subset\R^n$ have finite perimeter and let $u\in  S^{n-1}$ be given. Then
	\[
	\mathcal{SM}_{[0,u]}(\underline{A})=\int_{S^{n-1}}\langle u,v\rangle^+\, S_{n-1}^*(A,dv). 
	\]
	In other words, 
	\[
	\mathcal{SM}_{[0,u]}(A)=\int_{S^{n-1}}\langle u,v\rangle^+\, S_{n-1}^*(A,dv)
	\]
	holds if $A\cap A^0=\emptyset$. 
\end{theorem}

\begin{proof} 
	Consider the mapping
	$$g:(a,t)\mapsto a+tu,\quad (a,t)\in\cF_{u+}A\times[0,r].$$
	$g$ is a Lipschitz mapping defined on a countably $\cH^n$-rectifiable set and its Jacobian is $J_ng(a,t)=\langle u,\nu_A(a)\rangle>0$, $(a,t)\in\cF_{u+}A\times[0,r]$. Thus, using the coarea formula (cf.\ \cite[Theorem~2.93]{AFP}) we obtain for the image $\ima g$ of $g$
	\begin{equation}\label{eq:im_gInt}
		\cH^{n}(\ima g)\leq\int_{\cF_{u+}A\times[0,r]}J_ng\, d\cH^n=r\int_{S^{n-1}} \langle u,v\rangle^+\, S_{n-1}^*(A,dv).
	\end{equation}
	Now consider the set $A^*\in [A]_\sim$ in Lemma \ref{L1}. Up to a set of $\lambda_n$-measure zero, the set $(A^*\oplus r[0,u])\setminus A^*$ is contained in  the image of $g$ due to \eqref{eq:Agood}. Thus \eqref{eq:im_gInt} implies   
	\begin{equation} \label{Jan*}
		\lambda_n\big((A^*\oplus r[0,u])\setminus A^*\big)\leq r\int_{S^{n-1}}\langle u,v\rangle^+\, S^*_{n-1}(A,dv),\quad r>0. 
	\end{equation}
	Since
	\[
	\lambda_n\big((\underline A \oplus r[0,u])\setminus \underline A\big)\le 
	\lambda_n\big((A^*\oplus r[0,u]) \setminus A^*\big), 
	\] 
	by 	Lemma \ref{G*}~(ii),  inequality \eqref{Jan*} also holds with $\underline A$ replacing $A^*$. Dividing by $r$ and taking limits yields 
	$$\limsup_{r\to 0_+}\frac 1r\lambda_n\big((\underline A\oplus r[0,u])\setminus \underline A\big)\leq\int_{S^{n-1}}\langle u,v\rangle^+\, S_{n-1}^*(A,dv). $$
	Since Theorem \ref{prop0}~(i) with $Q=[0,u]$ implies the reverse inequality for the limit inferior, this concludes 
	the proof.   
\end{proof}

\section{Sufficient conditions for convex structuring elements}
\label{sec:suff_cond}	

Recall that \eqref{def_M} defined the \emph{$Q$-Minkowski content} of a set $E\subset\R^n$ with structural element $Q$  as
\begin{equation*}  %\label{def:MQE}
	\mathcal{M}_{Q}(E)=\lim_{r\to 0_+}\frac 1{2r}\lambda_n(E\oplus rQ).
\end{equation*}
Note that $\mathcal{M}_{Q}(E)=\tfrac 1{2}\mathcal{SM}_{Q}(E)$ whenever $\lambda_n(E)=0$.
For $t>0$, the homogeneity property 
\begin{equation}\label{eq:homogen}
	\cM_{tQ}(E)=t\cM_Q(E)
\end{equation}
holds whenever at least one side is well-defined. 
Already  simple examples, which can be treated using classical convex geometry (cf.~\cite{Schneider}), for instance the boundary $E$ of a square in $\R^2$ and a triangle $Q\subset\R^2$, suggest that 
\begin{equation}\label{eq:MQlimit1}
	\mathcal{M}_{Q}(E)=\int_{E} h(\di Q,\nu_E(x))\cH^{n-1}(dx) 
\end{equation}
is the correct form. Here, $\nu_E(x)$ is one of the unit normals of  $E$ at $x$, and
\[
\di Q=\frac{Q\oplus (-Q)}{2}
\]
is the  \emph{symmetral} of $Q$. In particular, $\di Q$ is origin-symmetric, i.e.~$\di Q=-(\di Q)$. 
Roughly speaking, in the non-isotropic case, the contribution of a boundary patch is weighted by half the width of the structuring element $Q$ in the normal direction of the patch. 

Equation \eqref{eq:MQlimit1} also holds under weaker smoothness conditions on $E$ if $Q=B^n$:  it is sufficient that  $E$ is  a countably $\cH^{n-1}$-rectifiable compact set that satisfies the AFP-condition. In order to generalize this result to general convex $Q$, we start with a slight generalization of 
\cite[Theorem~2.104]{AFP}.

\begin{lemma}\label{lem2}
	Let $Q\subset\R^n$ be a nonempty compact convex set of dimension $k\in\{0,1,\ldots,n\}$ and let $E\subset\R^n$ be a countably $\cH^{n-1}$-rectifiable set. Then
	\begin{equation}
		\label{eq:lowerBound}
		\liminf_{r\to 0_+}\frac 1{2r}\lambda_n(E\oplus rQ)\ge \int_E h(\di Q,\nu_E(x))\,\cH^{n-1}(dx),
	\end{equation}
	where $\nu_E(x)$ is a unit vector perpendicular to $\Tan^{n-1}(E,x)$. 
\end{lemma}
We recall that $\nu_E(x)$ is uniquely 
determined up to sign at $\cH^{n-1}$-almost all $x\in E$. Hence, the right-hand side of \eqref{eq:lowerBound} is well-defined as $h(\di Q,\cdot)$ is an even function. 

\begin{proof} We may assume that $Q$ contains the origin. 
	The result is known if $\dim Q=n$, see \cite[Eq.~(3.2)]{LV16}. If $k<n$ we apply Fubini's theorem 
	$$\lambda_n(E\oplus rQ)=\int_{L^\perp}\lambda_k((E\cap (L+y))\oplus r Q)\,\lambda_{n-k}(dy),$$
	where $L$ is the linear span  of $Q$.
	Using Fatou's lemma and then \cite[Eq.~(3.2)]{LV16} in $L+y$, we get
	\begin{align*}
		\liminf_{r\to 0_+}\frac 1{2r}\lambda_n(E\oplus rQ)
		&\geq\int_{L^\perp}\liminf_{r\to 0_+}
		\frac 1{2r}\lambda_k((E\cap (L+y))\oplus rQ)\,\lambda_{n-k}(dy)\\
		&=\int_{L^\perp} \int_{E\cap(L+y)}h(\di Q,\nu_{E\cap(L+y)}^*(x))\,\cH^{k-1}(dx),
	\end{align*}
	where $\nu_{E\cap(L+y)}^*(x)$ is calculated with $L+y$ as ambient space. 
	Applying the coarea formula for $p_{L^\perp}$ on $E$ with Jacobian $J_{n-1}p_{L^\perp}(x)=\|p_L\nu_E(x)\|$, and the relation
	\[
	h(\di Q,\nu_E(x))=\|p_L\nu_{E}(x)\|\,h(\di Q,\nu_{E\cap(L+y)}^*(x))
	\]
	valid for $\lambda_{n-k}$-almost all $y\in L^\perp$ and $\cH^{k-1}$-almost all $x\in E\cap (L+y)$, we obtain the desired formula.
\end{proof}

\begin{lemma}
	\label{lem:goodsets}
	Let $Q\subset\R^n$ be a nonempty compact convex set  of dimension $k\in \{0,\ldots,n\}$,  and let $E\subset\R^n$ be a compact 
	$(n-1)$-rectifiable set.  Then 
	$$\cM_Q(E)= \int_E h(\di Q,\nu_E(x))\,\cH^{n-1}(dx).$$
\end{lemma}

\begin{proof}
	In view of Lemma \ref{lem2},  it is sufficient to show that
	\begin{equation}  \label{E_limsup}
		\limsup_{r\to 0_+}\frac 1{2r}\lambda_n(E\oplus rQ)\leq \int_E h(\di Q,\nu_E(x))\,\cH^{n-1}(dx). 
	\end{equation}
    If $\dim Q=n$ then \eqref{E_limsup} follows from \cite[Theorem~3.7]{LV16}. In the general case, we approximate a lower-dimensional set $Q$ by $Q\oplus\ep B^n$ just as we did in the proof of Theorem \ref{prop0}~(ii).
\end{proof}

\begin{theorem}  \label{TM}
	Let $Q\subset\R^n$ be a nonempty compact convex set contained in a  $k$-dimensional subspace $L$, $k\in \{0,\ldots,n\}$,  and let $E\subset\R^n$ be a  countably $\cH^{n-1}$-rectifiable compact set which satisfies the AFP-condition relative to $L$.
	
	Then $E$ admits the $Q$-Minkowski content 
	$$\cM_Q(E)=\int_E h(\di Q,\nu_E(x))\,\cH^{n-1}(dx).$$
\end{theorem}

\begin{proof}
	In view of Lemma \ref{lem2} it is enough to show that
	$$\limsup_{r\to 0_+}\frac 1{2r}\lambda_n(E\oplus rQ)\leq\int_{E}h(\di Q,\nu_E(x))\,\cH^{n-1}(dx)=:I_Q(E).$$
	Since $Q$ is $k$-dimensional, we may assume that 
	the origin is a relative interior point of $Q$ by translating the structuring element appropriately, if necessary.
	Due to \eqref{eq:homogen}, we may assume that $Q$ is contained 
	in the $k$-dimensional unit ball $B^k=B^n\cap L$. 	   
	Hence, there is a radius $r_0\in(0,1]$ such that $r_0B^k\subset Q\subset B^k$.
	
	Let $\ep\in (0,1)$ be given. We will make use of 
	the cylinder $C_{\varepsilon}= Q\oplus \varepsilon B^{n-k}$ with $B^{n-k}=B^n\cap L^\perp$.
	In view of \cite[\S 3.2.18]{Federer69}, 
	there exist disjoint  compact and $\cH^{n-1}$-rectifiable subsets 
	$E_1,E_2,\dots$ of $E$ with $0=\cH^{n-1}(E\setminus\bigcup_iE_i)
	=\nu(E\setminus\bigcup_iE_i)$.  
	Since $E$ is compact, the continuity of $\nu$ from above implies that there exists a number $N\in \N$ such that
	\begin{equation}\label{eq:outside_e0_irrelevant}
		\nu(E\setminus E_0)<(r_0\ep)^k, 	
	\end{equation} 
	where $E_0:=E_1\cup\dots\cup E_N$. 
	
	Fix $0<\lambda<1$,  and define, for any given $r>0$, the set
	\[
	\tilde{E}=\tilde{E}_{\varepsilon,\lambda r}:=
	E\setminus(E_0\oplus \lambda r C_{\varepsilon}).
	\]
	By Besikovitch's theorem \cite[Theorem~2.18]{AFP} there exist balls $B(x_j,r_0\ep\lambda r)$, $j\in J$, 
	with centres $x_j\in\tilde{E}$ covering $\tilde{E}$ and such that any given point in $\R^n$ is covered by at most $\xi$  balls (where $\xi$ depends on the dimension $n$ only). 
	Using the AFP-condition relative to $L$, we obtain
	\begin{align*}
		\sum_{j\in J}\gamma(r_0\ep\lambda r)^{k-1}\lambda_{n-k}(p_{L^\perp}(E\cap B(x_j,r_0\ep\lambda r)))&\leq\sum_{j\in J}\nu(B(x_j,r_0\ep\lambda r))\\
		&\leq\xi\nu\Big(\bigcup_{j\in J}B(x_j,r_0\ep\lambda r)\Big). 
	\end{align*}
	There is a constant $r(\varepsilon,\lambda,r_0)$ such that for  $0<r<r(\varepsilon,\lambda,r_0)$ none of the balls $B(x_j,r_0\ep\lambda r)$, $j\in J$ hits $E_0$. This can be seen from 
	\[
	B^n\subset B^k\oplus B^{n-k}\subset \frac1{r_0\varepsilon}C_\varepsilon,  
	\]
	and implies 
	\begin{align*}
		\sum_{j\in J}\gamma(r_0\ep\lambda r)^{k-1}\lambda_{n-k}(p_{L^\perp}(E\cap B(x_j,r_0\ep\lambda r)))
		&\leq\xi\nu\left((E\oplus r_0\ep\lambda rB^n)\setminus E_0\right)\\
		&\leq\xi(\nu(E\setminus E_0)+(r_0\ep)^k)\\
		&\leq 2\xi(r_0\ep)^k 
	\end{align*}
	for all sufficiently small $r$, where we have used 
	the continuity of the measure $\nu$ from above and 	\eqref{eq:outside_e0_irrelevant}. As a consequence, we obtain
	\begin{align*}
		\lambda_n(\tilde{E}\oplus rQ)
		&\le \lambda_n(\tilde{E}\oplus rB^k)\\
		&\leq\sum_{j\in J}\lambda_n\left((E\cap B(x_j,r_0\ep\lambda r))\oplus rB^k\right)\\
		&\leq\sum_{j\in J}\lambda_{n-k}(p_{L^\perp}(E\cap B(x_j,r_0\ep\lambda r)))\kappa_k((1+r_0\ep\lambda)r)^k\\
		&\leq 2^{k+1} \xi\kappa_k\gamma^{-1}\frac{\ep}{\lambda^{k-1}}r.
	\end{align*}
	We will now show the inclusion
	\begin{equation}\label{eq:notp111}
		E\oplus rQ\subset [\tilde E\oplus rQ]\cup 
		[E_0\oplus (1+\lambda)rC_\varepsilon]. 
	\end{equation}
	Indeed, if $x\in E\oplus rQ$ but $x\not \in \tilde E\oplus rQ$ then there is an $x'\in E\setminus \tilde E$ such that $x\in x'+rQ$. 
	Since $E\setminus \tilde E\subset E_0\oplus \lambda r C_\varepsilon$, we arrive at $x\in E_0\oplus \lambda r C_\varepsilon\oplus rQ\subset 
	E_0\oplus (1+\lambda) r C_\varepsilon$ and assertion
	\eqref{eq:notp111} is proven. 
	
	Inclusion \eqref{eq:notp111}, 
	together with Lemma \ref{lem:goodsets}, applied to $E_1,\ldots,E_N$,  now shows that 
	\begin{align*}
		\limsup_{r\to 0_+}&\frac 1{2r}\lambda_n(E\oplus rQ)
		\\&\leq  2^{k} \xi\kappa_k\gamma^{-1}\frac{\ep}{\lambda^{k-1}}
		+(1+\lambda)I_{C_\varepsilon}(E_0)
		\\&\leq  2^{k} \xi\kappa_k\gamma^{-1}\frac{\ep}{\lambda^{k-1}}
		+(1+\lambda)I_{C_\varepsilon}(E).
	\end{align*}
	At the second inequality sign, we used the facts that $h(\di C_\varepsilon,\cdot)\ge 0$ and  $E_0\subset E$, implying that $\nu_E(x)$ and $\nu_{E_0}(x)$ coincide for $\cH^{n-1}$-almost $x\in E_0$; see Subsection~\ref{GMT}. 
	Setting $\lambda:=\ep^{1/n}$ and letting $\ep\to 0$ yields 
	\begin{align*}
		\limsup_{r\to 0_+}\frac 1r\lambda_n(E\oplus rQ)
		&\leq \int_{E} h(\di Q,\nu_E(x))\,\cH^{n-1}(dx), 
	\end{align*}
	completing the proof.
\end{proof}

It is worth mentioning that the above proof is a
generalization of that of \cite[Theorem~2.104]{AFP}. However, \eqref{eq:notp111} replaces a similar, but incorrect  inclusion in the first displayed formula on  \cite[p.~111]{AFP}. 

The following proposition is an extension of \cite[Theorem~5]{ACV08} for the outer $Q$-Minkowski content with lower-dimensional structuring element $Q$. 
	
\begin{proposition}  \label{prop_AFP} 
    Let $Q\subset\R^n$ be a nonempty compact convex set, 
    %of dimension $k<n$, $L\subset\R^n$ the linear $k$-subspace parallel to $Q$ 
    and let $A\subset\R^n$ be a set with finite perimeter. If $\partial A$ admits the $Q$-Minkowski content and this equals 
    $P_{\di Q}(A)$ then $A$ admits the outer $Q$-Minkowski content, and this is given by
    \begin{equation}
	\mathcal{SM}_{Q}(A)=P_Q(A). 
	\label{eq:LVresult}
    \end{equation}
\end{proposition}

\begin{proof}	
    The proof uses the same idea as that of \cite[Theorem~5]{ACV08}. 
    Define $a(r):=\frac 1rG(rQ,\1_A)$ and $b(r):=\frac 1rG(rQ,\1_{A^C})$, $r>0$. We know from Theorem~\ref{prop0}~(i) that
	\begin{align}
		\liminf_{r\to 0+} a(r)&\geq a:=P_Q(A), \label{li1}\\
		\liminf_{r\to 0+} b(r)&\geq b:=P_{Q}(A^C)=P_{-Q}(A), \label{li2}
	\end{align}
    where \eqref{eq_S_AC} was used at the last equality sign. 
    We can assume without loss of generality that $0\in Q$.
    %\mk{that $0$ is a  point in the relative interior $\relint Q$ of $Q$}. 
    Let $r>0$ be given. Note that if $x\in (A\oplus rQ)\setminus A$ then $x\in A^C$ and $(x-rQ)\cap A\neq\emptyset$, hence $(x-rQ)\cap \partial A\neq\emptyset$, by the convexity of $Q$. The same holds true if $x\in (A^C\oplus rQ)\setminus A^C$. 
    Thus we have the inclusion
	$$\left[(A\oplus rQ)\setminus A\right]\cup \left[(A^C\oplus rQ)\setminus A^C\right]\subset\partial A\oplus rQ.$$
    As the union of the sets in brackets is disjoint, we obtain
    $$G(rQ,\1_A)+G(rQ,1_{A^C})\leq\lambda_n(\partial A\oplus rQ),\quad r>0.$$
    Our assumption that $\partial A$ admits the $Q$-Minkowski content equal 
    to $P_{\di Q}(A)$ now yields
    $$\limsup_{r\to 0+}(a(r)+b(r))\leq 2\cM^{n-1}_Q(\partial A)=2P_{\di Q}(A)=P_Q(A)+P_{-Q}(A)=a+b.$$
    But then, due to \eqref{li1} and \eqref{li2}, it follows that  $\lim_{r\to 0+}a(r)=a$ and $\lim_{r\to 0+}b(r)=b$ (cf.\ \cite[Lemma~1]{ACV08}).
\end{proof}

Combining Proposition~\ref{prop_AFP} with Theorem~\ref{TM}, we obtain
	
\begin{corollary}  \label{C1}
    Let $Q\subset\R^n$ be a nonempty compact convex set of dimension $k<n$, $L\subset\R^n$ the linear $k$-subspace parallel to $Q$ and $A\subset\R^n$ a set with finite perimeter such that 
    $P(A)=\cH^{n-1}(\partial A)$.
    If $\partial A$ is compact and satisfies the  AFP-condition relative to $L$,
    then $A$ admits the outer $Q$-Minkowski content, and this is given by \eqref{eq:LVresult}.
\end{corollary}

\begin{proof}
    First, note that the assumption $P(A)=\cH^{n-1}(\partial A)$ implies $\cH^{n-1}(\partial A\setminus\partial^*A)=0$, see \eqref{perimeter}. Since $\partial^*A$ is countably $\cH^{n-1}$-rectifiable (see \cite[Theorems~3.59, 3.61]{AFP}), $\partial A$ is countably $\cH^{n-1}$-rectifiable as well. Applying Theorem~\ref{TM} to $\partial A$, we obtain the existence of 
    $$\cM_Q(\partial A)=\int_{\partial A}h(\di Q,\nu_{\partial A}(x))\, \cH^{n-1}(dx).$$
    Using that $\cH^{n-1}(\partial A\setminus\partial^*A)=0$, we have $\nu_{\partial A}(x)=\pm\Delta_A(x)$ $\cH^{n-1}$-almost everywhere on $\partial A$ (cf.\ Subsection~\ref{GMT}) and, as $\di Q$ is symmetric, we have
    $$P_{\di Q}(A)=\int_{\partial A}h(\di Q,\nu_{\partial A}(x))\, \cH^{n-1}(dx).$$
    Now we can apply Proposition~\ref{prop_AFP} and obtain the result.
\end{proof}

\section{Examples}  \label{Sect_examples}

\paragraph{Example 1}
Given $n\ge 4$ and $2\leq k\leq n-2$ we give a simple example of a set $A\subset\R^n$ with finite perimeter which does not admit the outer $B^n$-Minkowski content, 
but admits the outer $Q$-Minkowski content for some $k$-dimensional unit ball $Q\subset\R^n$.

Let $C\subset\R^k$ be a set with finite and positive volume, finite perimeter, and admitting the outer $B^k$-Minkowski content, and let $D\subset\R^{n-k}$ be another set with finite and positive volume, finite perimeter and such that 
\[\limsup_{r\to 0_+}\frac 1r\lambda_{n-k}((D\oplus rB^{n-k})\setminus D)=\infty.
\]
Hence, $D$ does not admit the outer $B^{n-k}$-Minkowski content; such a set is, for instance, constructed in \cite[Example~3]{ACV08}. The set  $A:=C\times D$ has finite perimeter
$$P(A)=P(C)\lambda_{n-k}(D)+\lambda_k(C)P(D),$$
and
\begin{align*}
	\lambda_n\big((A\oplus rB^n)\setminus A\big)
	&\geq\lambda_n\big((C\times (D\oplus rB^{n-k}))\setminus(C\times D)\big)\\
	&=\lambda_n(C\times((D\oplus rB^{n-k})\setminus D))\\
	&=\lambda_k(C)\lambda_{n-k}((D\oplus rB^{n-k})\setminus D),
\end{align*}
hence, $\limsup_{r\to 0_+}\frac 1r\lambda_n((A\oplus rB^n)\setminus A)=\infty$ by our assumption, so $A$ does not admit the outer $B^n$-Minkowski content. On the other hand, we have
\begin{align*}
	\lambda_n\big((A\oplus r(B^k\times\{0\}))\setminus A\big)
	&=\lambda_n\big(((C\oplus rB^k)\times D)\setminus(C\times D)\big)\\
	&=\lambda_n\big(((C\oplus rB^k)\setminus C)\times D\big)\\
	&=\lambda_k\big((C\oplus rB^k)\setminus C\big)\lambda_{n-k}(D).
\end{align*}
This, together with the choice of $C$ implies 
\begin{align*}
	\lim_{r\to 0_+}\frac 1r\lambda_n((A\oplus r(B^k\times\{0\}))\setminus A)&=P(C)\lambda_{n-k}(D)\\
	&=\int h(B^k\times\{0\},v)\, S^*_A(dv),
\end{align*}
hence, $A$ admits the outer $(B^k\times\{0\})$-Minkowski content.

\paragraph{Example 2}
We construct a compact set $A\subset\R^3$  of finite perimeter not admitting  the isotropic outer Minkowski content, nor the outer $Q$-Minkowski content for any two-dimensional convex body $Q\subset \R^3$ with the origin in the affine hull of $Q$.

Let  $D\subset\R^2$ be a compact set with finite perimeter such that $\tfrac1r G(rB^2,1_D)\to \infty$ as $r\to 0_+$ (in particular, $D$ does not admit the isotropic Minkowski content in the plane; see again~\cite[Example~3]{ACV08}). Consider the set $A':=D\times[0,1]$. As $A'\oplus r(B^2\oplus\{0\})=(D\oplus rB^2)\times[0,1]$, $A'$ does not admit the outer $(B^2\oplus\{0\})$-Minkowski content and, hence, neither the isotropic outer Minkowski content.  If  $L\subset \R^3$ is any linear $2$-subspace not containing $e_3$,
and $Q$ is a two-dimensional convex body contained in it, there is a $\rho>0$ such that $B^L=B^3\cap L\subset \rho^{-1}Q$, so 
$$A'\oplus rQ\supset A'\oplus r\rho B^L\supset(D\oplus (r\rho\cos\alpha)B^2)\times[r\rho\sin\alpha,1-r\rho\sin\alpha],$$
holds for all sufficiently small $r>0$, 
where $\alpha:=\angle(L^\perp,e_3)$. This implies that $A'$ does not admit the outer $B^L$-Minkowski content.

If $A$ is a  disjoint union of three properly chosen isometric copies of  $A'$, then $A$ does not allow for outer $Q$-Minkowski content for any two-dimensional compact convex set  $Q\subset L$, where $L$ is some two-dimensional subspace. 

\paragraph{Example 3}
We construct now a set $A\subset\R^3$ with finite perimeter which does not admit the outer $B^3$-Minkowski content, but admits the outer $B^2$-Minkowski content for \emph{any} two-dimensional unit ball $B^2$ in $\R^3$. The example could be adapted to higher dimensions.

Let $\delta:[0,\infty)\to[0,\infty)$ be a strictly increasing $C^1$-function with $\delta(0)=0$, $\delta(t)=o(t^3)$ ($t\to 0_+$) and $\delta'(t)\leq\frac{1}{32}$, $t\ge0$. (We can take for example  $\delta(t)=\delta_0\exp(-\frac 1t)$ with sufficiently small $\delta_0>0$, or choose $k>3$ and consider $\delta(t)=t^k/(32k)$ for $t\le 1$ and $\delta(t)=(1/k-1+t)/32$, otherwise.) Abusing  the notation slightly, we shall also write  $\delta(x):=\delta(\|x\|)$, $x\in \R^3\setminus\{0\}$.

Let $S\subset B^3\setminus\{0\}$ be a \emph{maximal} set such that
\begin{equation}\label{Ex1}
	\bigcup_{x\in S}B(x,\delta(x))\text{ is a disjoint union,}
\end{equation}
and observe that
\begin{equation}\label{Ex2}
	\bigcup_{x\in S}B(x,3\delta(x))\supset B^3\setminus\{0\}.
\end{equation}
Indeed, for any $y\in B^3\setminus\{0\}$ there exists $x\in S$ with $B(x,\delta(x))\cap B(y,\delta(y))\neq\emptyset$ (from the maximality of $S$), thus $\delta(x)+\delta(y)\geq\|x-y\|$. Further, as $\delta$ is $({1}/{32})$-Lipschitz, we have $\delta(x)-\delta(y)\geq-\frac{1}{32}\|x-y\|$, and, summing up both inequalities, we get $2\delta(x)\geq\frac{31}{32}\|x-y\|$, hence $y\in B(x,3\delta(x))$.

We shall further show that for any measurable set $D\subset B^3$  and any measurable function $f:\R^3\setminus \{0\}\to[0,\infty)$,
\begin{equation}\label{Ex3}
	\frac{1}{36\pi}\int_{D^{-3\delta}}f^{-3\delta}(y)\, dy
	\leq\sum_{x\in D\cap S}f(x)\delta(x)^3\leq
	\frac{3}{4\pi}\int_{D^{+\delta}}f^{+\delta}(y)\, dy,
\end{equation}
where $$D^{+\alpha\delta}:=\bigcup_{x\in D}B(x,\alpha\delta(x)),\quad  D^{-\alpha\delta}
:=D\setminus\bigcup_{x\in \R^3\setminus(\{0\}\cup D)}B(x,\alpha\delta(x)),$$
and 
\begin{align*}
	f^{+\alpha\delta}(y)&:=\sup\big\{f(z):\, \|z-y\|\leq \alpha\delta(y)\big\},\\ f^{-\alpha\delta}(y)&:=\inf\big\{f(z):\, \|z-y\|\leq \alpha\delta(y)\big\},\quad \alpha>0. 
\end{align*}
 The second inequality in \eqref{Ex3} follows from \eqref{Ex1} since
\begin{align*}
	\frac{4\pi}{3}\sum_{x\in D\cap S}f(x)\delta(x)^3&
	=\sum_{x\in D\cap S}f(x)\lambda_3\big(B(x,\delta(x))\big)\\
	&\leq \sum_{x\in D\cap S}\int_{B(x,\delta(x))}f^{+\delta}(y)\, dy\\
	&\leq \int_{D^{+\delta}}f^{+\delta}(y)\, dy.
\end{align*}
For the first inequality in \eqref{Ex3} we use \eqref{Ex2} yielding $D^{-3\delta}\subset\bigcup_{x\in D\cap S}B(x,3\delta(x))$, hence
\begin{align*}
	\int_{D^{-3\delta}}f^{-3\delta}(y)\, dy&\leq\sum_{x\in D\cap S}\int_{B(x,3\delta(x))}f^{-3\delta}(y)\, dy\\
	&\leq\sum_{x\in D\cap S}f(x)\lambda_3(B(x,3\delta(x)))\\
	&=36\pi\sum_{x\in D\cap S}f(x)\delta(x)^3.
\end{align*}
Set $\rho(x):=\delta(x)^3$ and consider the 
 union of disjoint balls
$$A:=\bigcup_{x\in S}B(x,\rho(x)).$$
Inclusion  \eqref{Ex2} implies 
\begin{equation} \label{Ex4}
	A\oplus B(0,3\delta(t))\supset S\oplus B(0,3\delta(t))\supset B(0,t-3\delta(t))\setminus\{0\},\quad 0<t\leq 1.
\end{equation}
Using the fact that
$$\rho(t)\leq\delta(t)\leq\delta(1)\leq\frac {1}{32},\quad t\leq 1,$$
and \eqref{Ex3}, we get
\begin{align*}
	\lambda_3\big(A\cap B(0,t-3\delta(t))\big)&\leq\sum_{x\in S\cap B(0,t-2\delta(t))}\frac{4\pi}{3}\rho(x)^3\\
	&\leq\frac{4\pi}{3} \frac{1}{32^6}\sum_{x\in S\cap B(0,t-2\delta(t))}\delta(x)^3\\
	&\leq  \frac{1}{32^6} \lambda_3\big(B(0,t-\delta(t))\big).
\end{align*}
This and \eqref{Ex4} yield together
\begin{align*}
	\frac{\lambda_3\big((A\oplus 3\delta(t)B^3)\setminus A\big)}{3\delta(t)} &\geq\frac{\frac{4\pi}{3}(t-3\delta(t))^3-\frac{1}{32^6}\frac{4\pi}{3}(t-\delta(t))^3}{3\delta(t)}\\
	&=\frac{4\pi}{9}\frac{32^6-1}{32^6}\frac{t^3}{\delta(t)}+o\left(\frac{t^3}{\delta(t)}\right)\to\infty,\quad t\to 0_+.
\end{align*}
Hence,
\begin{equation*} %\label{Ex5}
	\lim_{r\to 0_+}\frac 1rG(rB^3,1_A)=\infty,
\end{equation*}
so $A$ does not admit the isotropic outer Minkowski content.
Again by \eqref{Ex3}, we have
\begin{align*}
	P(A)&=\sum_{x\in S}4\pi\rho(x)^2=\sum_{x\in S}4\pi\delta(x)^6\\
	&\leq 3\int_{B(0,1+\delta(1))}(\delta(y)^3)^{+\delta}\, dy\\
	&\leq 3\lambda_3(B(0,1+\delta(1)))<\infty,
\end{align*}
so $A$ has finite perimeter.

Let now $B^2$ be any two-dimensional unit ball in $\R^3$.
An easy computation yields that for any $\rho,r>0$,
$$\lambda_3((\rho B^3\oplus rB^2)\setminus(\rho B^3))=\pi^2\rho^2r+2\pi\rho r^2.$$
Thus we have for any $0<r,t<1$
\begin{align}\nonumber
	{\frac1r}\lambda_3((A\oplus rB^2)\setminus A) \leq&\sum_{x\in S\setminus B(0,t)}\pi^2\rho(x)^2+\sum_{x\in S\setminus B(0,t)}2\pi\rho(x)r\\
	&+\frac1r\sum_{x\in S\cap B(0,t))}2\rho(x)\pi(r+t)^2,\label{eq:janGood}
\end{align}
where we used the fact that if $x\in B(0,t)$ then $B(x,\rho(x))\oplus rB^2$ is contained in a cylinder with height $2\rho(x)$ and whose base is the disc of radius $t+r$. The first sum is  exactly the anisotropic perimeter $P_{B^2}(A_t)$
of the set $A_t=\bigcup_{x\in S\setminus B(0,t)}B(x,\rho(x))$, which is a finite union of disjoint balls (since $\rho(\|x\|)\ge \rho(t)>0$ for all $t\le \|x\|\le 1$).  Note that $P_{B^2}(A_t)$ and thus the first sum in \eqref{eq:janGood} is bounded from above by the anisotropic perimeter $	P_{B^2}(A)$ of $A$. 
Thus, for any $0<r<1$ we have
\begin{align*}
	\frac 1rG(rB^2,\1_A)&\leq P_{B^2}(A)+\sum_{x\in S\setminus B(0,t))}2\pi\rho(x)r+\sum_{x\in S\cap B(0,t)}2\pi\rho(x)\frac 1r(r+t)^2.
\end{align*}
Set $b(t):=\sum_{x\in S\cap B(0,t)}\rho(x)$; we have, using \eqref{Ex3},
$$b(t)\leq\frac{3}{4\pi}\lambda_3\big(B(0,t+\delta(t))\big)\to 0,\quad t\to 0_+,$$
so choosing $r:=b(t)$, we get
$$\frac 1rB(rB^2,\1_A)\leq P_{B^2}(A)+2\pi b(1)b(t)+2\pi (b(t)+t)^2\to P_{B^2}(A),$$
as $t\to 0_+$. This, together with  Theorem \ref{prop0}(i) implies that $A$ allows for the outer $B^2$-Minkowski content.\medskip 

Finally, we shall show that $\partial A$ satisfies the AFP-condition relative to any two-dimensional linear subspace, so Corollary~\ref{C1} can successfully be applied here, since $P(A)=\cH^2(\partial A)$. Set
$$\nu(D):=\sum_{x\in S}\frac{1}{\rho(x)}\cH^2(D\cap\partial B(x,\rho(x)))$$
for any Borel set $D\subset\R^3$.
The measure $\nu$ is clearly absolutely continuous with respect to $\cH^2$. We shall show that $\nu$ is finite (hence Radon). Using \eqref{Ex3} we get
$$\nu(\R^3)\leq\sum_{x\in S}4\pi\rho(x)=\sum_{x\in S}4\pi\delta^3(x)\leq 3\lambda^3(B(1+\delta(1)))<\infty.$$

Let $L\subset\R^3$ be any $2$-subspace and $B^2\subset L$ be its unit ball. We shall verify condition \eqref{suff_k} for any ball $B(a,r)$, $a\in\partial A$, $r\in(0,1)$. We will use repeatedly the fact that if $a,x\in \R^2$ and $r>0$ then 
\begin{align}
	\label{inclusions}\begin{array}{rcl}
	B(a,r)\cap B(x,\delta(x))\ne \emptyset &\Rightarrow& x\in B(a,r)^{+(32/31)\delta},
	\\
	x\in B(a,r)^{-(32/31)\delta}&\Rightarrow &B(x,\delta(x))\subset B(a,r),
	\end{array} 
\end{align}	
where we used the notation introduced after \eqref{Ex3}.  
Indeed, to show the first implication suppose that $y\in B(x,\delta(x))\cap B(a,r)$. The Lipschitz property of $\delta$ then gives 
$
\delta(x)\le \delta(y)+\tfrac1{32}\|x-y\|\,\leq\, \delta(y)+\tfrac1{32}\delta(x),
$
so $\delta(x)\le (32/31)\delta(y)$ giving $\|x-y\|\le \delta(x)\le (32/31)\delta(y)$, as required. The contraposition of the second statement is shown with the same argument.
A similar use of the Lipschitz continuity of $\delta$ implies also that for all $a,x\in \R^2$, $s>0$ and $\alpha>0$ we have 
\begin{align}
	\label{moreInclusions}
    B(a,s)^{+\alpha\delta}\subset B\big(a,(1+\tfrac{\alpha}{32})s+\alpha \delta(a)\big),
\end{align}
If in addition, the radius on the left side of the following inclusion is positive, we also have
\begin{align}
	\label{moreInclusions2}
 B\big(a,(1-\tfrac{\alpha}{32})s-\alpha\delta(a)\big) \subset B(a,s)^{-\alpha\delta}.
\end{align}	
We only give arguments for the second of these inclusions.
Indeed, if it were not the case, there would be a point $x$ in the ball on the left-hand side of \eqref{moreInclusions2} and a point 
$y$ with $\|y-a\|=r'>s$ such that $\|x-y\|\le \alpha\delta(y)\le\alpha(\delta(a)+r'/32)$. Hence,
\[
\|a-x\|+\|x-y\|\le (1-\tfrac{\alpha}{32})s-\alpha\delta(a)+\alpha\delta(a)+\tfrac{\alpha}{32}r'<r'=\|a-y\|,
\]
which is a contradiction. 

To show the claim, we will now distinguish two cases.

\paragraph{Case $16\delta(a)\le r< 1$.} 
If $x\in S$  is such that $B(x,\rho(x))\cap B(a,r)\ne \emptyset$, then $\rho\le1$, \eqref{inclusions}, \eqref{moreInclusions} and $\delta(a)\le r/16$ give
\[
x\in B\big(a,\tfrac{32}{31}(\delta(a)+r)\big)\subset B\big(a,\tfrac{34}{31}r\big)
\]
The inclusion \eqref{moreInclusions} also gives %for any $0<r\le 31/34(1-\|a\|)$ 
\[
B\big(a,\tfrac{34}{31}r\big)^{+\delta}\subset 
B(a,\delta(a)+\tfrac{33}{31}\tfrac{17}{16}r)\subset 
B(a,\tfrac1{16}+\tfrac{33}{31}\tfrac{17}{16}r)
\subset B(a,2r),
\]
where the last inclusion comes from a crude estimate, which however is sufficient for the purpose.

Applying  \eqref{Ex3} with $f\equiv 1$ we thus get
\begin{align*}
	\lambda_1(p_{L^\perp}(\partial A\cap B(a,r)))&\leq\sum_{x\in S\cap   B(a, \frac{34}{31} r)}2\rho(x)\leq\sum_{x\in S\cap   B(a, \frac{34}{31} r)}2\delta(x)^3
    \\
	&\leq\frac{3}{2\pi}\lambda_3(B(a,2r))
    \\
%	&\leq\frac{3}{2\pi}\lambda_3(B(a,2r))\\
	&=16r^3.
\end{align*}
If, on the other hand, $x\in S\cap B(a, \frac{28}{31}r)$,  then $\rho\le1$, \eqref{inclusions}, \eqref{moreInclusions2} and $\delta(a)\le r/16$ imply  $B(x,\rho(x))\subset  B(a,r)$, so 
\eqref{Ex3} yields 
\begin{align*}
	\nu(B(a,r))&\geq\sum_{x\in S\cap B(a,\frac{28}{31}r)}\nu(\partial B(x,\rho(x)))\\
	&=4\pi\sum_{x\in S\cap B(a,\frac{28}{31}r)}\rho(x)
	=4\pi\sum_{x\in S\cap B(a,\frac{28}{31}r)}\delta(x)^3\\
	&\geq\tfrac{1}{9}\lambda_3\left(B(a,\tfrac{28}{31}r)^{-3\delta}\cap B(0,1)^{-3\delta}\right)
    \\
    &\geq \tfrac{1}{9}\lambda_3\left(B(a,\tfrac{28\cdot29}{31\cdot 32}r-3\delta(a))\cap B(0,1-3\delta(1))\right),
    \\
    &\geq \tfrac{1}{9}\lambda_3\left(B(a,\tfrac12 r)\cap B(0,1-3\delta(1))\right),
\end{align*}
where \eqref{moreInclusions2} 
was used for the  penultimate inequality and 
$\delta(a)\le r/16$ for the last 
inequality. We shall show that
\begin{equation}  \label{E_balls}
	\frac r4\geq \|a\|+3\delta(1)-1.
\end{equation}
Then the intersection of the two balls $B(a,\tfrac 12r)\cap B(0,1-3\delta(1))$ contains the circumball of the segment $[(\|a\|-\frac{r}{2})\frac{a}{\|a\|},(1-3\delta(1))\frac{a}{\|a\|}]$ of length at least $\frac r4$, hence 
$$\nu(B(a,r))\geq \frac{1}{9} \frac{4\pi}{3}\frac{1}{8^3}r^3.$$
Thus \eqref{suff_k} is satisfied with some $\gamma>0$ (we even do not need the factor $r$ on the right hand side in this case).

It remains to verify \eqref{E_balls}. This inequality is trivial when the right expression is negative, so we may assume $1-3\delta(1)\leq \|a\|\leq 1+\delta(1)$, where the right-hand bound is due to the fact that $a\in A$. Using the Lipschitz property of $\delta$, we get $\delta(1)\leq\delta(a)+\frac{1}{32}3\delta(1)$, hence $29\delta(1)\leq 32\delta(a)$, and, since $16\delta(a)\le r$ and $\|a\|\leq 1+\delta(1)$, the claim \eqref{E_balls} is obtained.

\paragraph{Case $0<r< 16\delta(a)$.} 
Since $a\in\partial A$, there exists $x_0\in S$ with $a\in\partial B(x_0,\rho(x_0))$, and we have
$$\nu(B(a,r))\geq\frac{1}{\rho(x_0)}\cH^2\left(\partial B(x_0,\rho(x_0))\cap B(a,r)\right).$$
Standard calculus with spherical coordinates yields that
$$\cH^2\left(\partial B(x_0,\rho(x_0))\cap B(a,r)\right)=\pi r^2\quad\text{ if } r\leq 2\rho(x_0).$$
Thus, as $\rho\leq 1$,
$$\nu(B(a,r))\geq\frac{1}{\rho(x_0)}\min\{\pi r^2,4\pi\rho(x_0)^2\}\geq \pi\min\{ r^2,\rho(x_0)\}.$$

Assume now that $B(a,r)\cap B(y,\rho(y))\neq\emptyset$ for some $y\in S$. In view of \eqref{inclusions} and \eqref{moreInclusions}, we have  
\[
y\in B(a,r)^{+2\delta}\subset B(a,\tfrac{17}{16}r+2\delta(a))\subset B(a,19\,\delta(a)). 
\]
Since there is an $x\in B(a,r)\cap B(y,\rho(y))$, the Lipschitz property of $\delta$ implies $|\delta(x)- \delta(a)|\le r/32$ and 
$|\delta(y)-\delta(x)|\le\delta(y)/32$, which gives
\begin{equation}
    \label{oneMore}
    \tfrac 13\delta(a)
    \le 
    \tfrac{32}{33}(\delta(a)-\tfrac{r}{32})\le \delta(y)\le \tfrac{32}{31}(\delta(a)+\tfrac{r}{32})\le 2\delta(a).
\end{equation} Hence, we get
\begin{align*}
	\lambda_1\big(p_{L^\perp}(\partial B(y,\rho(y))\cap B(a,r)\big)&
    \leq\diam\big(B(y,\rho(y))\cap B(a,r)\big)\\
	&\leq \min\{2r,2\rho(y)\}\leq 2\min\{r,2^3\rho(a)\}\\
	&\leq 2^7\min\{r,\rho(x_0)\}.
\end{align*}
where we also used the fact that $a\in B(x_0,\delta(x_0))$ implies $\delta(a)\le 2 \delta(x_0)$.
But the number of $x\in S$ such that $B(a,r)\cap B(x,\rho(x))\neq\emptyset$ is bounded by a universal constant in this case. Indeed, using \eqref{Ex3}, \eqref{moreInclusions} and \eqref{oneMore} again, we have
\begin{align*}
	 \card\big(S\cap B(a,19\,\delta(a))\big)
    &\leq\frac{3}{4\pi}\int_{B(a,19\,\delta(a))^{+\delta}}(\delta(y)^{-3})^{+\delta}\, dy\\
	&\leq \frac{3}{4\pi}{3}^3\frac{1}{\delta(a)^3}\lambda_3(B(a,19\,\delta(a))^{+\delta})\\
	&\leq   \frac{3}{4\pi}3^3\frac{1}{\delta(a)^3}\lambda_3(B(a,21\,\delta(a)))\\
	&\leq 3^3\,{21}^3=3^6\,7^3.
\end{align*}
Putting the two estimates together, we obtain (recall that $r\leq 1$)
$$r\lambda_1(p_{L^\perp}(\partial A\cap B(a,r)))\leq  2^7\,3^6\,7^3\min\{r^2,\rho(x_0)\},$$
and we find that again, \eqref{suff_k} is satisfied with some $\gamma>0$.
\medskip 

\emph{Acknowledgment.} The first author's  research was funded in part by Aarhus University Research Foundation, grant number AUFF-E-2024-6-13.
\medskip

%\emph{Data availability statement.} No datasets were generated or analyzed during the current study.
%\medskip

%\emph{Conflict-of-interest statement.} The authors have no competing interests to declare that are relevant to the content of this article.

\end{document}